\documentclass[12pt]{amsart}

\usepackage{mathrsfs, amssymb,amsthm, amsfonts, amsmath}
\usepackage[all]{xy}
\usepackage{upgreek}
\usepackage{amsmath}
\usepackage{textcomp}
\usepackage{hyperref}
\usepackage{setspace}

\newtheorem{theorem}[equation]{Theorem}
\newtheorem*{theorem*}{Theorem}
\newtheorem{lemma}[equation]{Lemma}
\newtheorem{corollary}[equation]{Corollary}

\newtheorem{question}[equation]{Question}

\theoremstyle{definition}
\newtheorem{example}[equation]{Example}
\newtheorem{definition}[equation]{Definition}
\newtheorem*{definition*}{Definition}

\theoremstyle{remark}
\newtheorem{remark}[equation]{Remark}

\makeatletter\@addtoreset{equation}{section}
\makeatother

\newcommand{\ZZ}{\mathbb{Z}}
\renewcommand{\P}{\mathbb{P}}
\newcommand{\PP}{\mathbb{P}}

\newcommand{\KK}{\mathbb{K}}
\newcommand{\OOO}{{\mathscr{O}}}
\newcommand{\LLL}{{\mathscr{L}}}
\newcommand{\EEE}{{\mathscr{E}}}

\newcommand{\Aut}{\operatorname{Aut}}

\newcommand{\Bim}{\operatorname{Bim}}

\newcommand{\Fix}{\operatorname{Fix}}

\def \ge {\geqslant}
\def \le {\leqslant}

\date{}

\title{Fiberwise bimeromorphic maps of conic bundles}

\author{Constantin Shramov}

\address{Steklov Mathematical Institute of Russian Academy of Sciences, 8 Gubkina st.,
Moscow, 119991, Russia
\newline
National Research University Higher School of Economics, Laboratory of Algebraic Geometry, 6 Usacheva str., Moscow, 119048, Russia
}

\email{costya.shramov@gmail.com}

\thanks{This work is supported by the Russian Science Foundation under grant \textnumero 18-11-00121.}

\begin{document}

\begin{abstract}
Given a holomorphic conic bundle without sections,
we show that the orders of
finite groups acting by its fiberwise bimeromorphic transformations
are bounded. This provides an analog of a similar result obtained
by T.\,Bandman and Yu.\,Zarhin for quasi-projective conic bundles.
\end{abstract}



\maketitle

\section{Introduction}

In many cases, biregular and birational structure of algebraic varieties and complex manifolds
is reflected in the properties of finite groups acting on them.
In particular, in certain situations there are boundedness results for such groups
that are implied by the properties of rational curves on these varieties.

The following theorem was proved in \cite{BandmanZarhin2017}
(see also~\cite[Corollary~4.12]{ShramovVologodsky} for a little bit more
general assertion).

\begin{theorem}[{\cite[Corollary~4.11]{BandmanZarhin2017}}]
\label{theorem:conic}
Let $\KK$ be a field of characteristic zero
that contains all roots of~$1$.
Let $C$ be a conic over $\KK$, and let $G\subset\Aut(X)$ be a finite subgroup.
Assume that $C$ has no $\KK$-points. Then
every non-trivial element of $G$ has order $2$,
and $G$ is either a trivial
subgroup, or is isomorphic to $\ZZ/2\ZZ$, or is isomorphic to $\ZZ/2\ZZ\times\ZZ/2\ZZ$.
\end{theorem}

Given a rational (or meromorphic) map $\phi\colon X\dasharrow Y$ and a
birational (or bimeromorphic) selfmap $g\colon X\dasharrow X$, we say that
$g$ is \emph{fiberwise with respect to $\phi$} if
for every point~$Q$ where
$g$ is defined one has $\phi(g(Q))=\phi(Q)$.
Applying Theorem~\ref{theorem:conic} to rational curve fibrations
without sections, one immediately obtains
the following.

\begin{corollary}[{see the proof of \cite[Theorem~1.5]{BandmanZarhin2017}}]
\label{corollary:conic-bundle}
Let $\Bbbk$ be an algebraically closed field of characteristic zero,
let $X$ and $Y$ be quasi-projective varieties over $\Bbbk$,
and let $\phi\colon X\to Y$ be a fibration whose general
fiber is a rational curve.
Let $G$ be a finite group acting on $X$ by birational maps
so that this action is fiberwise with respect to $\phi$.
Assume that $\phi$ has no
rational sections.
Then every non-trivial element of $G$ has order $2$,
and $G$ is either a trivial
subgroup, or is isomorphic to $\ZZ/2\ZZ$, or is isomorphic to $\ZZ/2\ZZ\times\ZZ/2\ZZ$.
\end{corollary}

To deduce Corollary~\ref{corollary:conic-bundle} from Theorem~\ref{theorem:conic}, one can note that
the group $G$ is a subgroup of the automorphism group
of the generic fiber $X_\eta$ of $\phi$, which is a conic over the function field
$\Bbbk(Y)$. Since $\phi$ has no rational sections, the conic $X_\eta$ has no $\Bbbk(Y)$-points,
and Theorem~\ref{theorem:conic} applies.

Apparently, there is no room for generalizations of Theorem~\ref{theorem:conic} in the case of complex manifolds.
However, the following result provides a natural generalization of
Corollary~\ref{corollary:conic-bundle}.

\begin{theorem}\label{theorem:conic-bundle}
Let $\phi\colon X\to Y$ be a proper surjective holomorphic map of irreducible complex manifolds whose
typical fiber is isomorphic to $\PP^1$.
Suppose that there does not exist a divisor on~$X$
whose intersection with a typical fiber of $\phi$ equals $1$.
Let $G$ be a finite group acting on~$X$ by bimeromorphic maps
so that this action is fiberwise with respect to $\phi$.
Then every non-trivial element of $G$ has order $2$. Moreover, $G$ is either a trivial
subgroup, or is isomorphic to $\ZZ/2\ZZ$, or is isomorphic to $\ZZ/2\ZZ\times\ZZ/2\ZZ$.
\end{theorem}

We also prove an analog of Theorem~\ref{theorem:conic-bundle} for
fiberwise automorphisms, see Lemma~\ref{lemma:non-decomposable-proj} below.

Similarly to \cite{BandmanZarhin2017} where Theorem~\ref{theorem:conic}
was applied to obtain further results concerning groups of birational
automorphisms of higher-dimensional varieties, we use Theorem~\ref{theorem:conic-bundle}
to obtain some results on bimeromorphic automorphisms of
higher-dimensional compact complex manifolds with a structure of a conic bundle.
We refer the reader to Corollaries~\ref{corollary:strongly-Jordan}, \ref{corollary:torus},
\ref{corollary:conic-bundle-section}, and~\ref{corollary:torus-Bim} below for details.

\smallskip
In~\S\ref{section:notation} we set the notation that will be used in the paper.
In~\S\ref{section:preliminaries} we collect several auxilairy results.
In~\S\ref{section:conic-bundles} we study fixed points of fiberwise bimeromorphic maps and prove Theorem~\ref{theorem:conic-bundle}.
In~\S\ref{section:Jordan} we apply Theorem~\ref{theorem:conic-bundle} to study Jordan property
for groups of bimeromorphic maps of certain complex manifolds.

\smallskip
I am grateful to T.\,Bandman, A.\,Efimov, A.\,Kuznetsov,
Yu.\,Prokhorov, and Yu.\,Zarhin for useful discussions.

\section{Notation}
\label{section:notation}

In this paper all complex manifolds are assumed to be irreducible.
We refer the reader to \cite{GPR}
for the basic facts on complex manifolds (and complex analytic spaces).

A proper surjective holomorphic map $h\colon X\to Y$ of complex manifolds
is called a \emph{modification} if
there exist closed analytic subsets $V\subsetneq X$ and $W\subsetneq Y$
such that $h$ restricts to a biholomorphic map $X\setminus V\to Y\setminus W$.
A \emph{meromorphic} map $f\colon X\dasharrow Y$ is defined by the closure of its graph $\Gamma_f\subset X\times Y$,
that is assumed to be a proper closed analytic subset of $X\times Y$ such that the projection
$\Gamma_f\to X$ is a modification. The meromorphic map $f$ is said to be
\emph{bimeromorphic} if the projection $\Gamma_f\to Y$ is a modification as well.
We refer the reader to \cite[\S7.1]{GPR} for more details concerning these definitions.

Let $\phi\colon X\to Y$ be a holomorphic map
of complex manifolds. By a \emph{fiber} of $\phi$ over a point~\mbox{$P\in Y$}
we mean the
(possibly non-reduced) complex analytic space $\phi^*P$.
For instance, if we say that all fibers of $\phi$ are isomorphic to $\PP^1$, we mean that
they are reduced and isomorphic to $\PP^1$ with the most usual reduced structure.
A  \emph{section} of $\phi$ is a closed analytic
subset~\mbox{$D\subset X$} that intersects every fiber of $\phi$
by a single point. By a \emph{$\PP^1$-bundle} we mean a holomorphic map
of complex manifolds whose fibers are
isomorphic to $\PP^1$; note that a $\PP^1$-bundle is automatically locally
trivial.

A holomorphic map of complex manifolds $\phi\colon X\to Y$ is called
\emph{proper} if
the preimage of any compact subset of $Y$ is compact. The map $\phi$ is proper if and only if
it is closed and has compact fibers. If $\phi$ is surjective, it is enough
to require that its fibers are compact.

Given a complex manifold $X$, by $\Bim(X)$ we denote its group of bimeromorphic
selfmaps, and by $\Aut(X)$ we denote its group of biholomorphic
selfmaps.
Given a holomorphic map~\mbox{$\phi\colon X\to Y$} of complex manifolds
and a bimeromorphic map $g\colon X\dasharrow X$, we say that $\phi$ is
\emph{$g$-equivariant} if there exist a bimeromorphic map
$g_Y\colon Y\dasharrow Y$ such that~\mbox{$\phi\circ g=g_Y\circ \phi$}
as meromorphic
maps from $X$ to $Y$.
By $\Bim(X;\phi)$ we denote the subgroup of~\mbox{$\Bim(X)$} that consists of all
selfmaps $g$ such that $\phi$ is $g$-equivariant.
By~\mbox{$\Bim(X)_\phi$} we denote the subgroup of $\Bim(X; \phi)$ that consists of all selfmaps
whose action is fiberwise with respect to $\phi$; more precisely, we require that for every point $Q$ where
$g$ is defined one has $\phi(g(Q))=\phi(Q)$.
We set
$$
\Aut(X;\phi)=\Bim(X;\phi)\cap\Aut(X)
$$
and~\mbox{$\Aut(X)_\phi=\Bim(X)_\phi\cap\Aut(X)$}.

For a meromorphic map $g\colon X\dasharrow X$,
we denote by $I(g)$ the set of points in $X$ where $g$ is not holomorphic.
By $\Fix(g)$ we denote the set of points $x\in X$ such that $g$ is holomorphic at $x$
and $g(x)=x$.

By $T_P(X)$ we denote the tangent space to a complex manifold $X$ at
a point $P\in X$.
By $\OOO_X$ we denote the sheaf of holomorphic functions on $X$.

A \emph{typical point} of a complex manifold $X$ is a  point outside a proper closed analytic
subset of $X$. A \emph{typical fiber} of a holomorphic (or meromorphic) map
is a fiber over a typical point in the image.

\section{Preliminaries}
\label{section:preliminaries}

In this section we collect several auxiliary results that will be used in the proof
of Theorem~\ref{theorem:conic-bundle} and related statements.

Recall that if $X$ is a complex manifold, $U$ is a dense open subset of
$X$, and $Z\subset U$ is a closed analytic subset of $U$, then the closure
$\bar{Z}$ of $Z$ in $X$ may be much larger than $Z$; in particular, $Z$ may be not dense
in $\bar{Z}$. However, for certain subsets the situation here is still
nice enough.

\begin{lemma}\label{lemma:fixed-points}
Let $X$ be a compact complex manifold,
and let $g\colon X\dasharrow X$ be a meromorphic map.
Then there exists a closed analytic subset
$\overline{\Fix}(g)$ of $X$ such that $\Fix(g)$ is a dense open subset
of $\overline{\Fix}(g)$.
\end{lemma}

\begin{proof}
Consider the graph $\Gamma_g\subset X\times X$ of the map $g$.
Let $\Delta\subset X\times X$ be the diagonal, and let $\mathrm{pr}_1\colon X\times X\to X$
be the projection to the first factor. Set
$$
\overline{\Fix}(g)=\mathrm{pr}_1\big(\Gamma_g\cap\Delta\big).
$$
Then $\overline{\Fix}(g)$ is a closed analytic subset of $X$.

The set $I(g)$ of indeterminacy points of $g$ is a closed analytic subset of $X$.
Let~\mbox{$U=X\setminus I(g)$}. Then $U$ is a dense open subset in $X$.
It remains to notice that
$$
\Fix(g)=\overline{\Fix}(g)\cap U,
$$
which implies that $\Fix(g)$ is open and dense in $\overline{\Fix}(g)$.
\end{proof}

\begin{lemma}\label{lemma:faithful}
Let $\phi\colon X\to Y$ be a
surjective holomorphic map of compact complex manifolds whose
typical fiber has dimension $1$.
Let $G\subset\Bim(X)_\phi$ be a finite subgroup.
Then~$G$ acts by holomorphic maps in a neighborhood of
a typical fiber of~$\phi$, and this action is faithful
on a typical fiber of~$\phi$.
\end{lemma}
\begin{proof}
Restricting to the preimage of the complement to a suitable proper closed
analytic subset in $Y$, we may assume that all fibers of $\phi$ are
one-dimensional, see~\mbox{\cite[Theorem~2.1.19]{GPR}}.
This implies that the map $\phi$ is open by~\mbox{\cite[Theorem~2.1.18]{GPR}}.

Choose a non-trivial element $g\in G$. Then the set $I(g)$ of indeterminacy points of~$g$ is a
closed analytic subset of codimension at least
$2$ in $X$, see for instance~\mbox{\cite[Remark~7.1.8(1)]{GPR}}.
Since the dimension of a fiber of $\phi$ is $1$,
the image~$\phi(I(g))$ is contained in a proper closed analytic subset $Z_{g}$ of~$Y$.
Furthermore, we know from Lemma~\ref{lemma:fixed-points} that
there exists a proper closed analytic subset~\mbox{$\overline{\Fix}(g)$} of $X$ that contains~\mbox{$\Fix(g)$}.
Since $\phi$ is open, the set
$$
U_g=\phi\big(X\setminus\overline{\Fix}(g)\big)
$$
is open in $Y$, so that the set $W_g=Y\setminus U_g$ is closed.
Therefore, the action of $g$ is holomorphic and non-trivial on the fibers of $\phi$ over all points
of~\mbox{$Y\setminus (Z_g\cup W_g)$}. Since the group $G$ is finite,
the assertion of the lemma follows.
\end{proof}

The following result is well-known, see for instance~\cite[Theorem~2.2.13]{GPR}.

\begin{theorem}\label{theorem:flat}
Let $\phi\colon X\to Y$ be a holomorphic map of complex manifolds.
Suppose that all fibers of $\phi$ have dimension equal to~\mbox{$\dim X-\dim Y$}.
Then $\phi$ is flat.
\end{theorem}

\begin{corollary}\label{corollary:flat}
Let $\phi\colon X\to Y$ be a proper holomorphic map of complex manifolds.
Suppose that all fibers of $\phi$ have dimension equal to~\mbox{$\dim X-\dim Y$}.
Let $\LLL$ be a vector bundle on $X$, and let $\EEE=\phi_*\LLL$.
Then $\EEE$ is a vector bundle on $Y$.
\end{corollary}
\begin{proof}
By construction, $\LLL$ is a flat coherent sheaf on $X$.
Hence $\EEE$ is a coherent sheaf on~$Y$, see for instance~\cite[Theorem~3.4.1]{GPR}.
Moreover, since $\phi$ is flat by Theorem~\ref{theorem:flat},
we conclude that~$\EEE$ is a flat coherent sheaf on $Y$, see~\cite[Proposition~2.2.6(2)]{GPR}.
This means that $\EEE$ is a vector
bundle on~$Y$, see~\cite[Proposition~2.2.6(3)]{GPR}.
\end{proof}

The following lemma will be used in~\S\ref{section:Jordan}.

\begin{lemma}\label{lemma:projectivization}
Let $X$ and $Y$ be complex manifolds, and let $\phi\colon X\to Y$ be a $\PP^1$-bundle.
Suppose that there is a divisor
on $X$ whose intersection number with
a fiber of $\phi$ equals~$1$.
Then $X$ is isomorphic to a projectivization of a rank $2$ vector bundle on~$Y$.
Moreover, if there exist two disjoint sections of $\phi$, then this vector bundle is decomposable.
\end{lemma}
\begin{proof}
Let $D$ be a divisor  on $X$ whose intersection number
with a fiber of $\phi$ equals $1$. It defines a line bundle $\LLL=\OOO_X(D)$ on $X$.
Let $\EEE=\phi_*\LLL$.
Then $\EEE$ is a vector bundle on~$Y$ by Corollary~\ref{corollary:flat}.
Since the degree of the restriction $\LLL\vert_F$ equals $1$, where $F\cong\PP^1$ is a fiber
of $\phi$, the rank of $\EEE$ equals~$2$.
There is a natural holomorphic map $\psi\colon X\to \P_Y(\EEE)$, which
commutes with the projection on $Y$ and induces isomorphisms on fibers.
Thus $\psi$ is an isomorphism.

Now suppose that there exist two disjoint sections $D_1$ and $D_2$ of $\phi$.
Then $D_1$ and $D_2$ are (effective) divisors on $X$.
Consider the Koszul resolution
of the intersection $D_1\cap D_2=\varnothing$:
$$
0\to \OOO_X(-D_1-D_2)\to\OOO_X(-D_1)\oplus\OOO_X(-D_2)\to \OOO_X\to 0.
$$
Tensoring it with $\OOO_X(D_1)$ and taking a push-forward by $\phi$ we get
$$
0\to \phi_*\OOO_X(-D_2)\to \phi_*\OOO_X\oplus\phi_*\OOO_X(D_1-D_2)\to\phi_*\OOO_X(D_1)\to R^1\phi_*\OOO_X(-D_2).
$$
Note that $\phi_*\OOO_X=\OOO_Y$. Since
$\OOO_X(-D_2)$ restricts to $F\cong\PP^1$ as $\OOO_{\PP^1}(-1)$, we know that
$$
\phi_*\OOO_X(-D_2)=R^1\phi_*\OOO_X(-D_2)=0.
$$
Therefore, we obtain an isomorphism
$$
\OOO_Y\oplus\phi_*\OOO_X(D_1-D_2)\stackrel{\sim}\longrightarrow\phi_*\OOO_X(D_1).
$$
Note that $\OOO_X(D_1-D_2)$ restricts to $F$ as $\OOO_F$, so that
$\phi_*\OOO_X(D_1-D_2)$ is a line bundle on $Y$.
This means that $\EEE=\phi_*\OOO_X(D_1)$ is a decomposable vector bundle
of rank $2$ on~$Y$. It remains to recall from the first part of the proof that $X$ is
isomorphic to the projectivization of~$\EEE$.
\end{proof}

Similarly to Lemma~\ref{lemma:projectivization} one proves the following.

\begin{lemma}\label{lemma:conic-bundle-section}
Let $X$ and $Y$ be complex manifolds.
Let $\phi\colon X\to Y$ be a proper holomorphic map whose fibers are one-dimensional and
whose typical fiber is isomorphic to~$\PP^1$.
Suppose that there is a divisor $D$ on $X$ such that the intersection number of $D$ with
a fiber of $\phi$ equals $1$.
Then $X$ is bimeromorphic to a projectivization of a rank $2$ vector bundle on~$Y$.
\end{lemma}
\begin{proof}
The divisor $D$ defines a line bundle $\LLL=\OOO_X(D)$ on $X$.
Let $F$ be a typical fiber of~$\phi$.
Then $F\cong\PP^1$, and the degree of the restriction of $\LLL$ to $F$ equals $1$.
Let~\mbox{$\EEE=\phi_*\LLL$}.
Similarly to the proof of Lemma~\ref{lemma:projectivization},
we see that $\EEE$ is a vector bundle of rank~$2$.
Furthermore, there is a natural holomorphic map $\psi\colon X\to \P_Y(\EEE)$, which restricts to
an isomorphism on the dense open subset of $X$ swept out by smooth fibers of $\phi$.
One can easily see that the map $\psi$ is bimeromorphic (and actually is a modification).
\end{proof}

The following result is well-known.

\begin{lemma}\label{lemma:Aut-torus}
Let $X$ be a compact complex manifold that does not contain rational curves, and let $g\colon X\dasharrow X$ be a meromorphic map.
Then $g$ is holomorphic.
\end{lemma}

\begin{proof}
Suppose that $g$ is not holomorphic. Consider the regularization of $g$ given by a sequence of blow ups
of smooth centers. This gives a commutative diagram
$$
\xymatrix{
& Z\ar@{->}[ld]_f\ar@{->}[rd]^h &\\
X\ar@{-->}[rr]^g && X
}
$$
The exceptional locus of $f$ is covered by curves isomorphic to $\PP^1$
(this follows for instance from \cite[Theorem~7.2.8]{GPR} or
\cite[Theorem~4.1]{Fischer}).
Since $g$ is not holomorphic, some of these curves are not mapped
to points by $h$. However, there are no non-trivial maps of $\PP^1$
to $X$,
which gives a contradiction.
\end{proof}

In particular, Lemma~\ref{lemma:Aut-torus} applies to the case when $X$
is a complex torus, because the latter does not contain rational curves.

We conclude this section by an elementary observation concerning automorphisms of the projective line.

\begin{lemma}\label{lemma:P1}
Let $g\in\Aut(\PP^1)$ be an element of finite order $n\ge 2$. Then
$g$ has exactly two fixed points on $\PP^1$. Let $P_1$ and $P_2$ be these points.
Then there is a primitive $n$-th root of unity $\zeta$
such that $g$ acts in the one-dimensional
tangent spaces $T_{P_1}(F)$ and $T_{P_2}(F)$ by~$\zeta$ and~$\zeta^{-1}$, respectively.
\end{lemma}
\begin{proof}
In appropriate homogeneous coordinates $x$ and $y$ on $\PP^1$, one can write the action of $g$ as
$$
(x:y)\mapsto (\zeta x: y),
$$
where $\zeta$ is a primitive root of $1$.
\end{proof}

\section{Proof of Theorem~\ref{theorem:conic-bundle}}
\label{section:conic-bundles}

In this section we study fixed points of fiberwise bimeromorphic maps and prove Theorem~\ref{theorem:conic-bundle}.

\begin{lemma}\label{lemma:two-sections}
Let $\phi\colon X\to Y$ be a
surjective holomorphic map of compact complex manifolds whose
typical fiber is isomorphic to $\PP^1$.
Let $g\in \Bim(X)_{\phi}$ be a bimeromorphic map.
Suppose that the order of $g$ is finite and is larger than $2$.
Then there exist two distinct irreducible effective divisors
on $X$ whose intersection with a typical fiber of $\phi$ equals $1$.
Moreover, if~$g$ is biholomorphic and $\phi$ is a $\PP^1$-bundle, then these divisors may be chosen to
be disjoint sections of~$\phi$.
\end{lemma}

\begin{proof}
By Lemma~\ref{lemma:fixed-points}
there exists a closed analytic subset
$\overline{\Fix}(g)$ of $X$ such that~\mbox{$\Fix(g)$} is a dense open subset
of $\overline{\Fix}(g)$.
Since $g$ is not the identity map, one has $\overline{\Fix}(g)\neq X$.
Let $\Sigma$ be the union of irreducible components
of $\overline{\Fix}(g)$ that have codimension $1$ in $X$ and are mapped surjectively on $Y$ by $\phi$.
Then $\Sigma$ is a (possibly zero) effective divisor on $X$.

Let $n>2$ be the order of $g$.
Let $F$ be a typical fiber of $\phi$, so that $F\cong\PP^1$.
By Lemma~\ref{lemma:faithful} we may assume that the action of $g$ on $F$ is holomorphic,
and the restriction of $g$ to $F$ has order $n$.
Therefore, $g$ has exactly two fixed points on $F$, say, $P_1$ and $P_2$.
This means that $\Sigma$ intersects a typical fiber of $\phi$ by two points.

By Lemma~\ref{lemma:P1} there is a primitive $n$-th root of unity~$\zeta$
such that $g$ acts in the one-dimensional
tangent spaces $T_{P_1}(F)$ and $T_{P_2}(F)$ by~$\zeta$ and~$\zeta^{-1}$, respectively.
Since~\mbox{$n>2$}, we see that $\zeta\neq\zeta^{-1}$. Hence $\Sigma$
splits as a union~\mbox{$\Sigma_\zeta\cup\Sigma_{\zeta^{-1}}$} of two divisors whose intersection number
with a typical fiber of $\phi$ equals $1$.

Now suppose that $g$ is biholomorphic and $\phi$ is a $\PP^1$-bundle. Let us show that the divisors $\Sigma_\zeta$ and $\Sigma_{\zeta^{-1}}$
are disjoint sections of $\phi$. By Lemma~\ref{lemma:P1} to do this it is enough to
check that $g$ acts by an automorphism of order $n$ on every fiber of $\phi$.

Suppose that this is not the case. Then, replacing $g$ by its suitable power if necessary,
we may assume that there is a fiber $F$ of $\phi$ such that the non-trivial
automorphism $g$ restricts to the identity map on $F$. Let $P$ be a point on $F$.
Then $g$ acts non-trivially on the tangent space $T_P(X)$,
see for instance \cite[\S2.2]{Akhiezer} or \cite[Corollary~4.2]{ProkhorovShramov-compact}.
On the other hand, $g$ acts trivially on the subspace $T_P(F)\subset T_P(X)$.
The morphism $\phi$ is a submersion by \cite[Theorem~2.1.14]{GPR}.
Therefore,
$$
d\phi\colon T_P(X)\longrightarrow T_{\phi(P)}(Y)
$$
is a surjective linear map whose kernel is identified
with $T_P(F)$, see for instance~\mbox{\cite[Remark~2.1.15]{GPR}}.
Moreover, the map~$d\phi$ is $g$-equivariant, where the corresponding
action of $g$ on $Y$ is taken to be trivial.
Since $g$ acts trivially on the tangent space $T_{\phi(P)}(Y)$, we conclude that
the action of $g$ on $T_P(X)$ is trivial as well.
The obtained contradiction completes the proof of the lemma.
\end{proof}

Now we prove Theorem~\ref{theorem:conic-bundle}.

\begin{proof}[Proof of Theorem~\ref{theorem:conic-bundle}]
Let $g$ be a non-trivial element of $G$.
By Lemma~\ref{lemma:two-sections}, the order of $g$ equals $2$.
Using Lemma~\ref{lemma:faithful}, we may assume that the action of $G$ on $F\cong\PP^1$ is holomorphic and faithful.
Since all non-trivial elements of $G$ have order $2$, we conclude that $G$ is a subgroup of
$\ZZ/2\ZZ\times\ZZ/2\ZZ$.
\end{proof}

\begin{remark}
It follows from the proof of Theorem~\ref{theorem:conic-bundle} that its assertion actually holds
under a slightly weaker assumption: it is enough to require that
there do not exist two distinct irreducible effective divisors on $X$ whose intersection numbers
with a typical fiber of $\phi$ equal $1$.
\end{remark}

\begin{remark}
If $\KK$ is a field of characteristic zero that contains all roots of $1$, then the proof
of  Lemma~\ref{lemma:P1} works for automorphisms of $\PP^1_{\KK}$. Thus one can prove
Corollary~\ref{corollary:conic-bundle} using the same argument as in the proof of
Theorem~\ref{theorem:conic-bundle}.
\end{remark}

One more consequence of Lemma~\ref{lemma:two-sections} is the following
analog of Theorem~\ref{theorem:conic-bundle}.

\begin{lemma}\label{lemma:non-decomposable-proj}
Let $X$ and $Y$ be compact
complex manifolds, and let $\phi\colon X\to Y$ be a $\PP^1$-bundle.
Suppose that $X$ is not a projectivization
of a decomposable vector bundle of rank~$2$ on $Y$.
Let $G$ be a finite subgroup of $\Aut(X)_\phi$.
Then every non-trivial element of $G$ has order~$2$. Moreover, $G$ is either a trivial
subgroup, or is isomorphic to~\mbox{$\ZZ/2\ZZ$}, or is isomorphic to~\mbox{$\ZZ/2\ZZ\times\ZZ/2\ZZ$}.
\end{lemma}
\begin{proof}
Let $g$ be a non-trivial element of $g$.
We know from Lemma~\ref{lemma:projectivization} that there do not exist two disjoint sections
of $\phi$. Hence the order of $g$ equals $2$ by Lemma~\ref{lemma:two-sections}.
Using Lemma~\ref{lemma:faithful} as in the proof of Theorem~\ref{theorem:conic-bundle},
we conclude that $G$ is a subgroup of~\mbox{$\ZZ/2\ZZ\times\ZZ/2\ZZ$}.
\end{proof}

\section{Jordan property}
\label{section:Jordan}

In this section we apply the previous results to study groups
of bimeromorphic selfmaps.

\begin{definition}[{see \cite[Definition~2.1]{Popov-Jordan}, \cite[Definition~1.1]{BandmanZarhin2017}}]
\label{definition:Jordan}
A group~$\Gamma$ is called \emph{Jordan}
if there is a constant~$J$ such that
for any finite subgroup $G\subset\Gamma$ there exists
a normal abelian subgroup~\mbox{$A\subset G$} of index at most~$J$.
We say that $\Gamma$ is \emph{strongly Jordan} if it is Jordan and
there exists a constant $R=R(\Gamma)$ such that every finite subgroup
of $\Gamma$ is generated by at most $R$ elements.
\end{definition}

\begin{example}
\label{example:torus}
Let $Y$ be a complex torus. Then the group $\Bim(Y)$ is strongly Jordan.
Indeed, we have $\Bim(Y)=\Aut(Y)$ by Lemma~\ref{lemma:Aut-torus}.
On the other hand, it is easy to see that the group~\mbox{$\Aut(Y)$} is Jordan, see for instance \cite[Corollary~8.7]{ProkhorovShramov-compact}.
By~\mbox{\cite[Theorem~1.3]{MiR}} this implies that
$\Aut(Y)$ is strongly Jordan.
\end{example}

Given a group $\Gamma$, we say that $\Gamma$ has
\emph{bounded finite subgroups} if there exists
a constant~\mbox{$B=B(\Gamma)$} such that,
for any finite subgroup	$G\subset\Gamma$, one has $|G|\le B$.
For the following group-theoretic result we refer the reader to \cite[Lemma~2.8]{ProkhorovShramov2014}
or~\mbox{\cite[Lemma~2.2]{BandmanZarhin2015}}.

\begin{lemma}
\label{lemma:group-theory}
Let
$$
1\longrightarrow\Gamma'\longrightarrow\Gamma\longrightarrow\Gamma''
$$
be an exact sequence of groups. Suppose that $\Gamma'$ has bounded finite subgroups and $\Gamma''$ is strongly Jordan.
Then $\Gamma$ is strongly Jordan.
\end{lemma}

Now we will derive some corollaries from Theorem~\ref{theorem:conic-bundle}.

\begin{corollary}\label{corollary:strongly-Jordan}
Let $X$ and $Y$ be compact complex manifolds, and let $\phi\colon X\to Y$ be a
surjective holomorphic map whose typical fiber is isomorphic to~$\PP^1$.
Suppose that there does not exist an effective divisor on $X$ whose intersection number
with a typical fiber of $\phi$ equals~$1$.
Suppose also that the group $\Bim(Y)$ is strongly Jordan.
Then the group~\mbox{$\Bim(X; \phi)$} is strongly Jordan.
\end{corollary}

\begin{proof}
One has an exact sequence of groups
\begin{equation*}
1\to\Bim(X)_\phi\to \Bim(X;\phi)\to\Bim(Y)
\end{equation*}
We know from Theorem~\ref{theorem:conic-bundle} that the group
$\Bim(X)_\phi$ has bounded finite subgroups.
Therefore, the group $\Bim(X; \phi)$ is strongly Jordan by Lemma~\ref{lemma:group-theory}.
\end{proof}

We can also prove an analog of Corollary~\ref{corollary:strongly-Jordan} for
automorphism groups.

\begin{corollary}\label{corollary:strongly-Jordan-Aut}
Let $X$ and $Y$ be compact complex manifolds, and let $\phi\colon X\to Y$ be a
$\PP^1$-bundle.
Suppose that $X$ is not a projectivization of a decomposable vector
bundle of rank $2$ on~$Y$.
Suppose also that the group $\Aut(Y)$ is strongly Jordan.
Then the group~\mbox{$\Aut(X; \phi)$} is strongly Jordan.
\end{corollary}

\begin{proof}
One has an exact sequence of groups
\begin{equation*}
1\to\Aut(X)_\phi\to \Aut(X;\phi)\to\Aut(Y)
\end{equation*}
By Lemma~\ref{lemma:non-decomposable-proj} the group
$\Aut(X)_\phi$ has bounded finite subgroups.
Therefore, the group~\mbox{$\Aut(X; \phi)$} is strongly Jordan by Lemma~\ref{lemma:group-theory}.
\end{proof}

One interesting application of Corollary~\ref{corollary:strongly-Jordan-Aut} concerns
$\PP^1$-bundles over complex tori.

\begin{remark}\label{remark:equivariant}
Let $Y$ be a complex torus, let $X$ be a compact complex manifold,
and let~\mbox{$\phi\colon X\to Y$} be a holomorphic map whose typical fiber is isomorphic to~$\PP^1$.
Since there are no non-trivial maps of $\PP^1$ to the complex torus $Y$, we see that the image of
a typical fiber of~$\phi$ under any bimeromorphic map of $X$ projects to a point in $Y$,
that is, it is again a fiber of~$\phi$. This means that $\phi$ is
$g$-equivariant with respect to any element $g\in\Bim(X)$,
so that~\mbox{$\Bim(X)=\Bim(X;\phi)$} and~\mbox{$\Aut(X)=\Aut(X;\phi)$}.
\end{remark}

Corollaries~\ref{corollary:strongly-Jordan} and~\ref{corollary:strongly-Jordan-Aut} imply the following result.

\begin{corollary}
\label{corollary:torus}
Let $Y$ be a complex torus, let $X$ be a compact complex manifold,
and let~\mbox{$\phi\colon X\to Y$} be a $\PP^1$-bundle.
If $X$ is not a projectivization of a rank $2$ vector bundle on $Y$,
then the group $\Bim(X)$ is strongly Jordan. If
$X$ is not a projectivization of a decomposable vector
bundle of rank $2$ on $Y$, then the group $\Aut(X)$ is strongly Jordan.
\end{corollary}

\begin{proof}
By Remark~\ref{remark:equivariant}, we have
$\Bim(X)=\Bim(X;\phi)$ and~\mbox{$\Aut(X)=\Aut(X;\phi)$}. Furthermore, according to Example~\ref{example:torus},
the group $\Bim(Y)$ is strongly Jordan.

If $X$ is not a projectivization of a rank $2$ vector bundle on $Y$,
by Lemma~\ref{lemma:projectivization} there does not
exist a divisor on $X$ whose intersection number
with a fiber of $\phi$ equals~$1$.
Therefore, the group $\Bim(X)$ is strongly Jordan by
Corollary~\ref{corollary:strongly-Jordan}.
Similarly, if $X$ is not a projectivization of a decomposable vector
bundle of rank $2$ on $Y$, then the group $\Aut(X)$ is strongly Jordan
by Corollary~\ref{corollary:strongly-Jordan-Aut}.
\end{proof}

For results concerning (the absence of) Jordan property for groups of bimeromorphic automorphisms
of projectivizations of vector bundles
of rank $2$ on complex tori, we refer the reader to~\mbox{\cite[Theorems~1.9, 1.10, and 1.12]{Zarhin2019}}.

Note that for a quasi-projective conic bundle $\phi\colon X\to Y$
existence of a section implies that $X$ is birational to $Y\times\PP^1$.
Therefore, we see from Corollary~\ref{corollary:conic-bundle}
that if a birational automorphism of finite order greater than $2$
acts on $X$ so that the action is fiberwise with respect to $\phi$, then
$X$ is birational to $Y\times\PP^1$ (cf.~\mbox{\cite[Theorem~1.5]{BandmanZarhin2017}}).
For complex manifolds  we deduce the following consequence of Lemma~\ref{lemma:conic-bundle-section}.

\begin{corollary}\label{corollary:conic-bundle-section}
Let $X$ and $Y$ be compact complex manifolds.
Let $\phi\colon X\to Y$ be a surjective holomorphic map whose fibers are one-dimensional and
whose typical fiber is isomorphic to~$\PP^1$.
Suppose that $X$ is not bimeromorphic to a projectivization of a rank $2$ vector bundle on~$Y$.
Suppose also that the group $\Bim(Y)$ is strongly Jordan.
Then the group~\mbox{$\Bim(X; \phi)$} is strongly Jordan.
\end{corollary}

\begin{proof}
By Lemma~\ref{lemma:conic-bundle-section}, there does not
exist a divisor on $X$ whose intersection number
with a fiber of $\phi$ equals $1$. Hence the group $\Bim(X; \phi)$ is strongly Jordan by
Corollary~\ref{corollary:strongly-Jordan}.
\end{proof}

\begin{corollary}
\label{corollary:torus-Bim}
Let $Y$ be a complex torus, let $X$ be a compact complex manifold,
and let~\mbox{$\phi\colon X\to Y$} be a surjective holomorphic map whose fibers are one-dimensional and
whose typical fiber is isomorphic to~$\PP^1$.
Suppose that $X$ is not bimeromorphic to a projectivization of a rank $2$ vector bundle on $Y$.
Then the group $\Bim(X)$ is strongly Jordan.
\end{corollary}

\begin{proof}
By Remark~\ref{remark:equivariant} we have
$\Bim(X)=\Bim(X;\phi)$, and by Example~\ref{example:torus}
the group~\mbox{$\Bim(Y)$} is strongly Jordan.
Therefore, the assertion follows from Corollary~\ref{corollary:conic-bundle-section}.
\end{proof}

It would be interesting to find out if
Theorem~\ref{theorem:conic-bundle}
or Corollaries~\ref{corollary:strongly-Jordan} and~\ref{corollary:conic-bundle-section}
can be generalized to the case of fibrations whose typical
fiber is a rational surface. We refer the reader to
\cite{ProkhorovShramov-3folds} and \cite{ShramovVologodsky}
for results of similar flavor concerning projective varieties.
Also, I do not know the answer
to the following question.

\begin{question}
Does there exist an indecomposable vector bundle of rank $2$ on a complex torus
such that for its projectivization $X$ the group $\Bim(X)$ is not strongly Jordan?
\end{question}

Finally, the following general question looks interesting and relevant to the
subject of this paper.

\begin{question}[{cf. \cite[Proposition~4.1]{EN}}]
Over which complex tori there exist a $\PP^1$-bundle that is not
a projectivization of a rank $2$ vector bundle?
\end{question}

\end{document}